\documentclass[reqno, oneside, 12pt]{amsart}
\usepackage[letterpaper]{geometry}
\usepackage{setspace}
\usepackage{appendix}
\usepackage{amsmath}
\usepackage{amssymb}
\usepackage{enumerate,enumitem}
\usepackage{verbatim}
\usepackage{mathrsfs}
\usepackage{mathtools}
\usepackage{hyperref}
\usepackage{amsthm, mathdots, multicol,placeins}
\usepackage{xcolor}
\usepackage{ upgreek }
\usepackage{ dsfont }

\usepackage{thmtools}
\usepackage{thm-restate}

\usepackage[all,cmtip]{xy}
\usepackage{ytableau}

\DeclareMathOperator{\Jac}{Jac}
\DeclareMathOperator{\USp}{USp}
\DeclareMathOperator{\SU}{SU}
\DeclareMathOperator{\Sp}{Sp}
\DeclareMathOperator{\GSp}{GSp}
\DeclareMathOperator{\U}{U}
\DeclareMathOperator{\GL}{GL}
\DeclareMathOperator{\SL}{SL}

\DeclareMathOperator{\ST}{ST}
\DeclareMathOperator{\Gal}{Gal}
\DeclareMathOperator{\End}{End}
\DeclareMathOperator{\Aut}{Aut}
\DeclareMathOperator{\diag}{diag}

\DeclareMathOperator{\TL}{TL}
\DeclareMathOperator{\AST}{AST}
\DeclareMathOperator{\Hg}{Hg}

\DeclareMathOperator{\LL}{L}
\DeclareMathOperator{\tr}{tr}

\DeclareMathOperator{\MT}{MT}

\newcommand{\Z}{\mathbb{Z}}
\newcommand{\Q}{\mathbb{Q}}

\newcommand{\C}{\mathbb{C}}

\newtheorem{theorem}{Theorem}[section]

\newtheorem{proposition}[theorem]{Proposition}

\newtheorem{lemma}[theorem]{Lemma}

\newtheorem{corollary}[theorem]{Corollary}
\newtheorem{conjecture}[theorem]{Conjecture}

\theoremstyle{remark}
\newtheorem*{remark*}{Remark}

\numberwithin{equation}{section}

\author{Heidi Goodson}
\address{Department of Mathematics, Brooklyn College; 2900 Bedford Avenue, Brooklyn, NY 11210 USA}
\email{heidi.goodson@brooklyn.cuny.edu}


\title{An exploration of degeneracy in abelian varieties of Fermat type}

\begin{document}

\maketitle

\begin{abstract}
    The term \emph{degenerate} is used to describe abelian varieties whose Hodge rings contain exceptional cycles -- Hodge cycles that are not generated by divisor classes. We can see the effect of the exceptional cycles on the structure of an abelian variety through its Mumford-Tate group, Hodge group, and Sato-Tate group. In this article we examine degeneracy through these different but related lenses. We  specialize to a family of abelian varieties of Fermat type, namely Jacobians of hyperelliptic curves of the form $y^2=x^m-1$. We prove that the Jacobian of the curve is degenerate whenever $m$ is an odd, composite integer. We explore the various forms of degeneracy for several examples, each illustrating different phenomena that can occur.
\end{abstract}

\section{Introduction}

We say that a complex abelian variety $A$ is \emph{nondegenerate} if its (complexified) Hodge ring is generated by divisor classes. In this article we are interested in studying abelian varieties that do not have this property. Abelian varieties whose Hodge rings are not generated by divisor classes are called \emph{degenerate}, and the Hodge cycles not coming from the divisor classes are called \emph{exceptional cycles}. 

While the definition of degeneracy is a statement about the Hodge ring, we can see the effects of degeneracy in  groups constructed from the Hodge structure of the abelian variety: the Mumford-Tate group, the Hodge group, and the Sato-Tate group. The Mumford-Tate group and Hodge group are related to the Hodge ring via an action on certain cohomology groups. The algebraic Sato-Tate conjecture and Mumford-Tate conjecture imply a relationship between the Sato-Tate group and the Mumford-Tate and Hodge groups. Thus, it is natural that  degeneracies of these groups should be intertwined with degeneracies in the Hodge ring.

In this paper we focus on nonsimple abelian varieties and study how the Hodge structures of their simple factors interact with each other. This is particularly interesting in the case where one of the factors is of type IV in the Albert's classification since the product of nondegenerate simple abelian varieties is also nondegenerate as long as none of the factors are of this type (see Theorem 0.1 of \cite{Hazama89}). Work of Moonen and Zarhin \cite{MoonenZarhin99} gives a complete classification of degeneracy in dimensions 4 and 5 that is based on the decomposition of an abelian variety into its simple factors. There are some explicit examples of degenerate abelian varieties in the literature where the focus is on the associated Hodge rings (see, for example, \cite{Aoki2002, Pohlmann1968, Shioda82}). Recent work of Lombardo \cite{Lombardo2021} examines degeneracy in the Mumford-Tate group of a dimension 4 Jacobian variety. This phenomenon does not occur in lower dimensions since the Hodge ring of an abelian variety of dimension $\leq 3$ is always generated by divisor classes.

In this paper we build on these results by considering how the degeneracy of the Hodge ring can been seen in the Mumford-Tate group,  Hodge group, and  Sato-Tate group within a particularly interesting family of abelian varieties of Fermat type: Jacobians of curves of the form $C_m\colon y^2=x^m-1$. This family of curves was studied by Shioda in \cite{Shioda82}, where he proved the Hodge conjecture for their Jacobian varieties. The Hodge conjecture claims that the Hodge cycles of an abelian variety are all algebraic cycles, and one way to verify this is by showing that all of the Hodge cycles are generated by divisor classes  and then using the fact that the divisor classes are all algebraic. Shioda, on the other hand, proves the Hodge conjecture for these Jacobian varieties without appealing to divisor classes, and in certain cases he proves that the Hodge rings have exceptional cycles. In this paper we prove the following extension of his result.

\begin{theorem}\label{thm:degeneracyxm}
Let $m$ be an odd composite integer. Then the Jacobian $\Jac(C_m)$ of the curve $y^2=x^m-1$ is degenerate in the sense that the Hodge ring contains exceptional cycles.
\end{theorem}
We prove this result by first showing that, for any odd primes primes $p\leq q$ such that the product $pq$ divides $m$, the Jacobian of $y^2=x^{pq}-1$ can be viewed as a factor of $\Jac(C_m)$ (note that this includes the case where $p=q$ and so $p^2$ divides $m$). We then apply Lemma \ref{lemma:exceptionalHodge}, where we prove that $\Jac(C_{pq})$ has exceptional Hodge cycles.

With the general result of Theorem \ref{thm:degeneracyxm} in hand, we aim to understand how  degeneracy in the Hodge ring appears as  degeneracy in the Mumford-Tate, Hodge, and Sato-Tate groups of the Jacobian varieties. The Mumford-Tate and Hodge groups of degenerate abelian varieties are, in some sense, ``smaller" than what we would see for a nondegenerate abelian variety. The Sato-Tate group of a degenerate abelian variety will have a ``smaller" identity component and, potentially, a component group that is ``larger" than what we would see for a nondegenerate CM abelian variety (i.e., one that is larger than the Galois group of the CM field over the base field).

Determining Sato-Tate groups of abelian varieties is the source of ongoing interest and work. There has been recent progress on computing Sato-Tate groups of nondegenerate abelian varieties (see, for example, \cite{Arora2018, EmoryGoodson2020,emorygoodsonnondegeneracy,FiteGonzalezLario2016, FKS2012,SutherlandGenus3, FiteLorenzoSutherland2018, GoodsonCatalan, LarioSomoza2018}). A nice property of nondegenerate abelian varieties is that the component group of the Sato-Tate group is isomorphic to $\Gal(K/F)$, where $K$ is the endomorphism field (i.e., the minimal extension over which all the endomorphisms of the abelian variety are defined) and $F$ is the field of definition of the abelian variety -- this is due to the fact that the Sato-Tate group of the abelian variety is connected over this field. However, the landscape for degenerate abelian varieties is still wide open. In particular, there are no examples in the literature where the component groups for these degenerate Sato-Tate groups are computed. One difficulty is that we may need a larger field than the endomorphism field in order to have a connected Sato-Tate group, and it's not clear what field is needed. In this paper we have several results regarding the identity component of the Sato-Tate group (see Propositions \ref{prop:STidentity9} and \ref{prop:STidentity15}) as well as the following conjecture regarding the full Sato-Tate group.

\begin{restatable*}{conjecture}{conjectureST}
{\label{conjectureST9}}
Let $\Jac(C_9)$ denote the Jacobian of the curve $y^2=x^9-1$. Up to conjugation in $\USp(8)$, the Sato-Tate group of $\Jac(C_9)$  is $$\ST(\Jac(C_9)) \simeq \left\langle \diag(U_1,U_2,U_3,\overline U_1U_2U_3), \gamma\right\rangle,$$
where $\gamma$ is defined in Equation \eqref{eqn:J9gamma} and $U_i$ is defined in the Notation and Conventions section. 
\end{restatable*}

In addition to this example, we explore degeneracy for other values of $m$, carefully examining the Mumford-Tate groups, the Hodge rings and Hodge groups, and the Sato-Tate groups of the varieties. We see different phenomena occurring for different values of $m$, which makes this an intriguing problem to study.

\subsection*{Organization of the paper}

In Section \ref{sec:background} we provide background information for the main objects of study. 
In Section \ref{sec:Degeneracy} we take a deep dive into degeneracy. We begin by generalizing a result of Lombardo \cite{Lombardo2021} for degeneracy in Mumford-Tate groups. We prove Lemma \ref{lemma:exceptionalHodge}, which states that $\Jac(C_{pq})$ has exceptional Hodge cycles for any odd primes $p,q$ (not necessarily distinct). We then explain how the Hodge ring and Hodge group are related to each other. Finally, we describe how degeneracy can affect the Sato-Tate group of an abelian variety.
In Section \ref{sec:JacobianSplitting} we prove results for the decompositions of the Jacobian varieties into simple factors and we prove Theorem \ref{thm:degeneracyxm}.
In Section \ref{sec:DegenerateExamples} we work on explicit examples that demonstrate various phenomena that can occur for degenerate abelian varieties. For each example, we examine the Mumford-Tate group, Hodge ring, and Hodge group. For $m=9$ and $m=15$, we also study the Sato-Tate groups. While the Sato-Tate group of $\Jac(C_{15})$ is somewhat mysterious, we show that it is less mysterious for $\Jac(C_{9})$. We make a conjecture for the full Sato-Tate group of $\Jac(C_{9})$ and provide moment statistics to support the conjecture.

\subsection*{Notation and conventions}
For any $K$-vector space $W$ and $K$-algebra $R$, let $W_R:=W\otimes_K R.$

The curve $y^2=x^m-1$ is denoted by $C_{m}$ and its Jacobian is denoted by $J_m$. We  write $\zeta_m$ for a primitive $m^{th}$ root of unity. 

Let $I$ denote the $2\times 2$ identity matrix and define the matrix
\begin{align*}
    J=\begin{pmatrix}0&1\\-1&0\end{pmatrix}.
\end{align*}
We embed $\U(1)$ in $\SU(2)$ via $u\mapsto U= \diag(u, \overline u)$ 
and, for any positive integer $n$,  define the following subgroup of the unitary symplectic group $\USp(2n)$
\begin{equation*}\label{eqn:U1n}
    \U(1)^n:=\left\langle \diag( U_1, U_2,\ldots, U_n)\;|\; U_i\in \U(1)\right\rangle.
\end{equation*}

\section{Background}\label{sec:background}

\subsection{Hodge Structures}\label{sec:backgroundHodgeStructure}

In this section, we follow the exposition in Chapters 1 and 17 of \cite{BirkenhakeLange2004}. 

Let $A$ be a projective abelian variety over $\mathbb C$. We denote the first homology group of $A$ by $V(A):=H_1(A,\Q)$ and its dual (the first cohomology group) by $V^*(A):=H^1(A,\Q)$. The complex vector space $V(A)_\C$ has a weight $-1$ Hodge structure, i.e., a decomposition $V(A)_\C= V(A)^{-1,0}\oplus V(A)^{0,-1}$ where  $\overline{V(A)^{-1,0}}=V(A)^{0,-1}$. This corresponds to the following weight 1 Hodge structure of its dual
$$V^*_\C =H(A)^{1,0}\oplus H(A)^{0,1},$$
where $V(A)^{-1,0}=H^{1,0}(A)^*$ and $V(A)^{0,-1}=H^{0,1}(A)^*$. The notation in the decomposition of $H^1(A,\Q)$ is defined by $H^{a,b}(A)=H^a(\Omega^b_A),$ where $\Omega_A^b$ is the sheaf of holomorphic $b$-forms on $A$. We can also define $H^{a,b}(A)$ by
\begin{align}\label{eqn:Hodgeab}
    H^{a,b}(A)&\simeq \bigwedge^a H^{1,0}(A) \otimes \bigwedge^b H^{0,1}(A).
\end{align}
Hodge structures of weight $n$, $H^n(A,\C)\simeq \bigwedge^n H^1(A,\C)$, satisfy $ H^n(A,\C)\simeq \bigoplus_{a+b=n} H^{a,b}(A).$

\subsection{Hodge and Mumford-Tate Groups}\label{sec:backgroundHodgeMT}

The Hodge structure of the previous section determines a representation $\mu_{\infty,A}: \mathbb G_{m,\C} \rightarrow \GL_{V_\C}$ acting as multiplication by $z$ on $V(A)^{-1,0}$ and trivially on $V(A)^{0,-1}$. 
With this setup, we define the {Mumford-Tate group} of $A$, denoted $\MT(A)$, to be the smallest $\Q$-algebraic subgroup of $\GL_{V}$  such that $\mu_{\infty,A}(\mathbb G_{m,\C})\subseteq \MT(A)_\C$. We define the Hodge group of $A$ to be the connected component of the identity of $\MT(A)\cap \SL_{V}$. 

The Hodge group can also be formed by restricting the representation $\mu_{\infty,A}$ to the (real Lie) circle group $\mathbb S^1:=\{z\in\C\mid |z|=1\}$. With this setup, the Hodge group is the smallest $\Q$-algebraic subgroup of $\GL_{V}$ such that  $\mu_{\infty,A}(\mathbb S^1)\subseteq \Hg(A)_\C$. The image of this restriction of $\mu_{\infty,A}$ lies in $\SL_{V_\C}$, and so the Hodge group is a $\Q$-algebraic subgroup of $\SL_{V}$. In fact, one can show that the image of the representation $\mu_{\infty,A}$ is contained in the symplectic group $\GSp_{V_\C}$, taken with respect to the symplectic form given by the block matrix $\diag(J,\ldots,J)$. Hence, $\MT(A)$ and $\Hg(A)$ are $\Q$-algebraic subgroups of $\GSp_{V}$ and $\Sp_{V}$, respectively.

In our work we will be interested in nonsimple abelian varieties. For $n\geq 1$, we can identify $\Hg(A^n)$ with $\Hg(A)$ and the action is performed diagonally on $V(A^n)=(V(A))^n$. More generally, for $n_1,n_2,\ldots, n_t\geq 1$, $\Hg(A_1^{n_1}\times A_2^{n_2}\times \cdots \times A_t^{n_t})   $ is isomorphic to $\Hg(A_1\times A_2\times \cdots \times A_t)$. Even more generally, if $A$ and $B$ are abelian varieties then $\Hg(A\times B)\subseteq \Hg(A)\times \Hg(B)$. In our work we will be interested in studying cases where the containment is strict.

\subsection{The Hodge Ring}\label{sec:backgroundHodgeRing}
In this section we use the notation found in \cite{Shioda82}.  We denote the (complexified) Hodge ring of $A$ by
$$\mathscr B^*(A):=\displaystyle\sum_{d=0}^{\dim(A)} \mathscr B^d(A), $$
where $\mathscr B^d(A)=(H^{2d}(A,\mathbb Q)\cap H^{d,d}(A))\otimes \mathbb C$ is the $\mathbb C$-span of Hodge cycles of codimension $d$ on $A$.  
The subring of $\mathscr B^*(A)$ generated by the divisor classes, i.e.  generated by $\mathscr B^1(A)$ , is
$$\mathscr D^*(A):=\displaystyle\sum_{d=0}^{\dim(A)} \mathscr D^d(A),$$ 
where $\mathscr D^d(A)$ is the $\mathbb C$-span of classes of intersection of $d$ divisors. 

The relationship between these spaces is relevant to the Hodge conjecture: Let $\mathscr C^d(A)$ be the subspace of $\mathscr B^d(A)$ generated by the classes of algebraic cycles on $A$ of codimension $d$. Then 
$$\mathscr D^d(A) \subseteq \mathscr C^d(A) \subseteq \mathscr B^d(A)$$
and the Hodge conjecture for $A$ asserts that every Hodge cycle is algebraic: $\mathscr C^d(A) = \mathscr B^d(A)$ for all $d$ (see \cite{Aoki2002, Shioda82}). One way to prove the Hodge conjecture in codimension $d$ is to prove the equality $\mathscr D^d(A) =\mathscr B^d(A)$. However this equality does not always hold, even when the Hodge conjecture is known to be true. In Section \ref{sec:DegeneracyHodgeRing} we will further study the relationship between $\mathscr D^d(A)$ and $\mathscr B^d(A)$.

\subsection{An $\ell$-adic construction of the Sato-Tate group}\label{sec:backgroundSTgroup}
We follow the exposition of \cite{EmoryGoodson2020} and \cite[Section 3.2]{SutherlandNotes}. See also \cite[Chapter 8]{SerreNXP}.

For any prime $\ell$, we define the Tate module $T_{\ell}:=\varprojlim_{n} A[\ell^n]$; this is a free $\mathbb{Z}_{\ell}$-module of rank $2g$. We define the rational Tate module $V_{\ell}:=T_{\ell}\otimes_{\mathbb{Z}} \mathbb{Q}$, which is a $\mathbb{Q}_{\ell}$-vector space of dimension $2g.$ The Galois action on the Tate module is given by an $\ell$-adic representation 
\begin{align}\label{eqn:artinrepresntation}
    \rho_{A,\ell}:\Gal(\overline{F}/F) \rightarrow \Aut(V_{\ell}) \cong \GL_{2g,\mathbb{Q}_{\ell}}.
\end{align}
The $\ell$-adic monodromy group of $A$, denoted $G_{A,\ell}$, is the Zariski closure of the image of this map in $\GL_{2g,\mathbb{Q}_{\ell}}$, and we define $G^{1}_{A,\ell}:=G_{A,\ell}\cap \Sp_{2g,\mathbb{Q}_{\ell}}$. We define the {Sato-Tate group} of $A$, denoted   $\ST(A)$, to be a maximal compact Lie subgroup of $G^{1}_{A,\ell}\otimes_{\mathbb{Q}_{\ell}}\mathbb{C}$ contained in $\USp(2g)$. 

It is conjectured that $\ST(A)$ is independent of the choice of the prime $\ell$ and of the embedding of $\mathbb{Q}_\ell$ in $\mathbb{C}$; this is known to be true in many cases, such as in dimension $\leq 3$ and for products of CM abelian varieties (see, for example, \cite{Banaszak2015, Joh17}). While the Sato-Tate group is a compact Lie group, it may not be connected. We denote  the connected component of the  identity  (also called the identity component) of $\ST(A)$ by  $\ST^0(A)$.

\subsection{The Mumford-Tate and Algebraic Sato-Tate Conjectures}\label{sec:backgroundMTConjecture}
Deligne showed in \cite{Deligne1982}  that the identity component of $G_{A,\ell}$ is always a subgroup of $\MT(A)_{\Q_\ell}$, and it is conjectured that the inclusion is actually an equality (see \cite{SutherlandNotes}).
\begin{conjecture}[Mumford-Tate conjecture]
The identity component of $G_{A,\ell}$ is equal to $\MT(A)_{\Q_\ell}$.
\end{conjecture}
Since the identity component of $G^{1}_{A,\ell}$ is a subgroup of $\Hg(A)_{\Q_\ell}$, the Mumford-Tate conjecture predicts that this inclusion is also an equality. 

The Mumford-Tate conjecture is known to be true for many examples of abelian varieties, including for abelian varieties of CM type (see, for example, \cite{Pohlmann1968, Yu2015}) and for all abelian varieties of dimension $g\leq 3$ \cite[Theorem 6.11]{Banaszak2015}. More generally, Commelin proved that if the Mumford-Tate conjecture is true for abelian varieties $A$ and $B$ defined over a finitely generated field then it is also true for the product $A\times B$ \cite{Commelin2019}.

When the Mumford-Tate conjecture holds, it follows from results of \cite[Section 3]{Banaszak2015} that, up to conjugation in $\USp(2g)$, the Mumford-Tate group and the Hodge group uniquely determine the identity component of the Sato-Tate group. Unfortunately we will not be able to gain information about the component group of the Sato-Tate group from these two groups. To address this, Banaszak and Kedlaya gave the following refinement of the Mumford-Tate conjecture in \cite{Banaszak2015}.

\begin{conjecture}[algebraic Sato-Tate conjecture]\label{conjec:AST}  There is an algebraic subgroup $\AST(A)$ of $\GSp_{2g}$ over $\mathbb Q$, called the {algebraic Sato-Tate group of $A$}, such that the connected component of the identity $\AST^0(A)$ is reductive and, for each prime $\ell$, $G^{1}_{A,\ell}=\AST(A)\otimes_{\mathbb Q} \mathbb Q_{\ell}$.
\end{conjecture}

When this conjecture holds, the {Sato-Tate group} of $A$ is a maximal compact Lie subgroup of $\AST(A)\otimes_\mathbb{Q} \mathbb{C}$ contained in $\USp(2g)$. 

There are many cases where the algebraic Sato-Tate conjecture is known to be true.  Banaszak and Kedlaya prove that the conjecture holds for all abelian varieties of dimension at most 3 \cite[Theorem 6.11]{Banaszak2015}, for many examples of simple abelian varieties  \cite[Theorem 6.9]{Banaszak2015}, and for all abelian varieties of CM type  \cite[Theorem 6.6]{Banaszak2015}.  There are examples of infinite families of higher dimensional Jacobian varieties for which the algebraic Sato-Tate conjecture is known to be true in \cite{EmoryGoodson2020, FS2016, GoodsonCatalan}. Furthermore, Cantoral-Farf\'an and Commelin \cite{CantoralCommelin2022} proved that the algebraic Sato-Tate conjecture holds whenever the Mumford-Tate conjecture holds for an abelian variety.

\subsection{Moment Statistics}
In Section \ref{sec:DegenerateExamples} we will compute moment sequences of the Sato-Tate groups of  Jacobian varieties. These moment statistics can be used to support the equidistribution statement of the generalized Sato-Tate conjecture by comparing them to moment statistics obtained for the traces $a_i$ in the normalized $L$-polynomial. The numerical moment statistics are an approximation since one can only ever compute them up to some prime.

The following background information has been adapted from \cite{EmoryGoodson2020, GoodsonCatalan}. We define the $n^{th}$ moment (centered at 0) of a probability density function to be the expected value of the $n$th power of the values, i.e. $M_n[X]=E[X^n]$. For independent variables $X$ and $Y$ we have $E[X+Y]=E[X]+E[Y]$ and $E[XY]=E[X]E[Y]$, which yield the following identity
\begin{align*}
     M_n[X_1+\cdots + X_m]=\sum_{\beta_1+\cdots+\beta_m=n} \binom{n}{\beta_1,\ldots, \beta_m}M_{\beta_1}[X_1]\cdots M_{\beta_m}[X_m],
\end{align*}
where the $X_i$ are independent.

In Section \ref{sec:DegenerateExamples}, we will work with the unitary group $\U(1)$ and consider the trace map  on $U\in \U(1)$ defined by $z:=\tr(U)=u+\overline{u}=2\cos(\theta)$, where $u=e^{i\theta}$.  From here we see that $dz=2\sin(\theta)d\theta$ and 
$$\mu_{\U(1)}= \frac1\pi \frac{dz}{\sqrt{4-z^2}}$$
is the pushforward of the Haar measure on  $\U(1)$ to the interval $[-2,2]$
(see \cite[Section 2]{SutherlandNotes}). We can deduce the following pushforward measure
\begin{equation*}\label{eqn:muU1}
    \mu_{\U(1)^k}= \prod_{i=1}^k \frac1\pi \frac{dz_i}{\sqrt{4-z_i^2}}.
\end{equation*} 

We can now define the moment sequence $M[\mu]$, where $\mu$ is a positive measure on some interval $I=[-d,d]$. The $n^{th}$ moment $M_n[\mu]$  is, by definition,   $\mu(\phi_n)$, where $\phi_n$ is the function $z\mapsto z^n$. This yields $M_n[\mu_{\U(1)}]  = \binom{n}{n/2}$, where $\binom{n}{n/2}:=0$ if $n$ is odd. Hence, $M[\mu_{\U(1)}] = (1,0,2, 0,6,0,20,0,\ldots).$ From here, we  take binomial convolutions to obtain
\begin{align*}
    M_n[\mu_{\U(1)^k}] &=\sum_{\beta_1+\cdots+\beta_k=n} \binom{n}{\beta_1,\ldots, \beta_k}M_{\beta_1}[\mu_{\U(1)}]\cdots M_{\beta_k}[\mu_{\U(1)}].
\end{align*}

For a dimension $g$ abelian variety, denote by $\mu_i$ the projection of the Haar measure onto the interval $\left[-\binom{2g}{i},\binom{2g}{i}\right]$, where $i\in\{1,2, \ldots, g\}$. We  compute $M_n[\mu_i]$ by averaging over the components of the Sato-Tate group.


\section{Degenerate Abelian Varieties}\label{sec:Degeneracy}

We  define an abelian variety $A$ to be \emph{degenerate} if its Hodge ring is not generated by divisor classes, i.e., the containment  $\mathscr D^*(A) \subset \mathscr B^*(A)$ is strict. Results of Hazama \cite{Hazama89} show that $A$ is stably nondegenerate if and only if the rank of the Hodge group is maximal ($A$ is stably nondegenerate if $\mathscr D^*(A^n) = \mathscr B^*(A^n)$ for all $n>0$), which implies a similar statement for the Mumford-Tate group. Hazama also proves an interesting result regarding products of abelian varieties: if $A$ and $B$ are both stably nondegenerate, then the only scenario in which $A\times B$ could be degenerate is if one of its simple factors is of type IV in the Albert's classification (see  \cite{Hazama89}). Recall that an abelian variety $X$ is of type IV if the center of $\End(X)\otimes_\Z \Q$ is a CM field $K$ with totally real subfield $K_0$ containing elements of $K$ that are fixed by the Rosati involution induced by the polarization of $X$ (see, for example, \cite[Section 21]{Mumford69}).

The main goal of this section is to investigate degeneracy via Mumford-Tate groups, Hodge rings, and Hodge groups. The results of Sections \ref{sec:DegneracyMT} and \ref{sec:DegeneracyHodgeGroup} are general, while the result in Section \ref{sec:DegeneracyHodgeRing} holds for Jacobians of the curves $y^2=x^m-1$. In Section \ref{sec:DegeneracyST} we will show how degeneracy can affect the Sato-Tate group of an abelian variety.

\subsection{Degeneracy via the Mumford-Tate group}\label{sec:DegneracyMT}

In this section we study degeneracy of nonsimple abelian varieties via their Mumford-Tate groups. Consider a product  $A_1\times A_2\times \cdots \times  A_n$ of nonisogenous (absolutely) simple abelian varieties. The question we wish to answer is: when is the Mumford-Tate group of this product smaller than expected? In particular, we wish to know when one of the canonical projections
$$\MT(A_1\times A_2\times \cdots \times  A_n)\longrightarrow \MT(A_i)$$
is an isomorphism or an isogeny. To answer this, we generalize results of Lombardo \cite{Lombardo2021}. In what follows, we focus on abelian varieties $A_i$ with CM.

\subsubsection{Notation}\label{sec:MTnotation}

Let $E$ be a CM field and let $L/\Q$ be the Galois closure of $E/\Q.$ Let $T_E=Res_{E/\Q}(\mathbb G_m)$ be the corresponding algebraic torus over $\Q$. A specific case of this is $T_E(\Q)=E^\times$. See Section 2.3.1 of \cite{Lombardo2021} for more details. 

Let $G=\Gal(L/\Q)$ and let $H$ be the subgroup of $G$ associated to $E$: $H=\Gal(L/E)$. We can identify a CM type $\Phi$ of $E$ with a subset of $H\setminus G$, and we define the set $\widetilde{\Phi}=\{g\in G\mid Hg\in \Phi\}$. In our work, the reflex field of $E$ will always satisfy $E^*=E$ though its CM-type $\Phi^*$ will usually be distinct from $\Phi$.

There are several important maps that we will carefully define. The first of these is the usual norm function $N\colon L\to E$ defined by $N(x) = \prod_{\sigma} \sigma(x)$, where the product is taken over all $\sigma$ in $\Gal(L/E)$. In the special case where $L=E$, this norm map is simply the identity.  This norm map induces a map on the associated tori.

We define two dual maps on the character groups of the tori $T_E$, denoted by $\widehat{T}_E$. The first is a map on characters associated to $E$:
\begin{align*}
    \phi^*\colon \widehat{T}_E &\to \widehat{T}_{E^*}\\
            [\beta]&\mapsto \sum_{\psi\in\Phi^*} [\psi\beta].\label{eqn:phi}
\end{align*}
Here $ [\beta]$ denotes the character of $T_E$ corresponding to the embedding $\beta\colon E\to \overline\Q$. 
The second map sends characters associated to $E^*$ to those associated to $L$:
\begin{align*}
    N^*\colon \widehat{T}_{E^*} &\to \widehat{T}_{L}\\
            [Hg]&\mapsto \sum_{h\in H} [hg].
\end{align*}
Here we are defining the map by considering the isomorphism $\widehat{T}_{E^*}\simeq \mathbb Z[H\setminus G]$ (see the exposition in \cite{Lombardo2021}). 
In our work, we will be interested in the composition of these maps: $N^* \phi^*\colon \widehat{T}_{E} \to \widehat{T}_{L}$.

\subsubsection{The canonical projection map}\label{sec:canonicalprojmap}

The work in this section closely follows Section 4 of \cite{Lombardo2021} where the results are stated for a product of two absolutely simple abelian varieties. Let $E_i$ denote the CM field of the absolutely simple abelian variety $A_i$ for $1\leq i\leq n$. We begin with a result for character maps.

\begin{lemma}\label{lemma:MTmaps1}
For $i=1,\ldots,n$, let $E_i$ be a CM field, $L_i$ be its Galois closure over $\Q$, and $L$ be a finite Galois extension of $\Q$ that contains each $L_i$.  The image of the map $$T_L \xrightarrow{(N_1, N_2, \ldots, N_n)} T_{E_1^*} \times T_{E_2^*}\times \cdots \times T_{E_n^*} \xrightarrow{(\phi_1,\phi_2, \ldots,\phi_n)} T_{E_1} \times T_{E_2} \times \cdots \times T_{E_n}$$ is $\MT(A_1 \times A_2\times \cdots \times A_n)$.
\end{lemma}
\begin{proof}
See the proof of Lemma 4.1 of \cite{Lombardo2021}, which does not depend on the number of factors in the product. 
\end{proof}

Now consider the projection  $\pi_i\colon \MT(A_1\times A_2\times \cdots \times A_n)\to \MT(A_i)$. 

\begin{lemma}\label{lemma:MTmaps2}
Let $\widehat{T}_i=\{0\}\times \cdots  \times \{0\} \times\widehat{T}_{E_i} \times \{0\}\times \cdots  \times \{0\}.$ The canonical projection $\pi_i$ is an isomorphism if and only if 
\[
\widehat{T}_i + \ker (N_1^* \phi_1^* + \cdots + N_n^*\phi_n^*) = \widehat{T}_{E_1} \times  \cdots \times \widehat{T}_{E_n}.
\]
Similarly, $\pi_i$ is an isogeny if and only if $\widehat{T}_i + \ker (N_1^* \phi_1^* + \cdots + N_n^*\phi_n^*)$ has finite index in the product $ \widehat{T}_{E_1} \times  \cdots \times \widehat{T}_{E_n}$.
\end{lemma}

\begin{proof}
This result is an easy generalization of Lemma 4.2 of \cite{Lombardo2021}, and so we include just enough detail in the proof to aid in understanding the result.

Using the notation and definitions in Section \ref{sec:MTnotation} and the results of Lemma \ref{lemma:MTmaps1}, we can build the following commutative diagram
\[\xymatrixcolsep{6pc}
\xymatrix{
T_L \ar@{->>}[r]^{(\phi_1 N_1,\,\ldots,\, \phi_n N_n)\;\;\;\;} \ar[d]^{(N_1, \ldots, N_n)}  & \MT(A_1\times\cdots \times A_n) \ar@{->>}[r]^{\pi_i}  \ar@{^{(}->}[d] & \MT(A_i) \ar@{^{(}->}[d] \\
T_{E_1^*} \times \cdots  \times T_{E_n^*}  \ar[r]_{(\phi_1, \, \ldots,\, \phi_n)} & T_{E_1} \times \cdots \times T_{E_n} \ar@{->>}[r]_{p_i} & T_{E_i}
}
\]
and its dual
\[\xymatrixcolsep{5.5pc}
\xymatrix{
\widehat{T}_L & \widehat \MT(A_1\times\cdots \times A_n) \ar@{_{(}->}[l]_{N_1^* \phi_1^* + \cdots+ N_n^* \phi_n^*\;\;\;\;}    & \widehat \MT(A_i) \ar@{_{(}->}[l]_{\pi_i^*}  \\
\widehat T_{E_1^*} \times \cdots\times \widehat T_{E_n^*} \ar[u]^{N_1^* +\cdots+ N_n^*}   & \widehat T_{E_1} \times \cdots \times \widehat T_{E_n} \ar@{->>}[u] \ar[l]^{(\phi_1^*,\,\ldots,\, \phi_n^*)} & \widehat T_{E_i} \ar@{_{(}->}[l]^{p_i^*} \ar@{->>}[u]
 }
\]
where $\pi_i, p_i$ are the canonical projections on the $i$-th factor, and we denote by $f^*$ the map induced on characters by a morphism $f$ of algebraic tori.

Proving that the (surjective) projection $\pi_i$ is  injective is equivalent to proving that $\pi_i^*$ is surjective. Similarly, proving that the projection $\pi_i$ is an isogeny is equivalent to proving that $\pi_i^*$ has finite cokernel. By diagram chasing, in both cases we need only consider the kernel and cokernel of $$\widehat T_{E_i}\longrightarrow \widehat T_{E_1}\times \cdots \times \widehat T_{E_n}\longrightarrow \widehat \MT(A_1\times \cdots \times A_n).$$

The kernel of the map $\prod_i\widehat T_{E_i}\longrightarrow \widehat \MT(A_1\times \cdots \times A_n)$ coincides with the kernel of $$\widehat{T}_{E_1} \times \cdots \times \widehat{T}_{E_n} \xrightarrow{N_1^* \phi_1^* + \cdots+ N_n^*\phi_n^*} \widehat{T}_L$$ since $\widehat{\MT}(A_1\times \cdots \times A_n) \to \widehat{T}_L$ is injective. 
Hence, $\widehat{\MT}(A_1\times \cdots \times A_n)$ can be identified with $({\widehat{T}_{E_1} \times \cdots \times \widehat{T}_{E_n}})/{ \ker (N_1^* \phi_1^* + \cdots+ N_n^*\phi_n^*) }.$ 

Under this identification, the map $\pi_i^*$ is surjective if and only if $p_i^*\left( \widehat{T}_{E_i} \right)$, which equals $\widehat{T}_i$,  surjects onto ${\widehat{T}_{E_1} \times\cdots\times \widehat{T}_{E_n}}/{ \ker (N_1^* \phi_1^* +\cdots+ N_n^*\phi_n^*) }$. Similarly,  $\pi_i^*$ is an isogeny if and only if $p_i^*\left( \widehat{T}_{E_i} \right)$ maps to $ {\widehat{T}_{E_1} \times\cdots\times \widehat{T}_{E_n}}/{ \ker (N_1^* \phi_1^* +\cdots+ N_n^*\phi_n^*) }$ with finite cokernel.
\end{proof}

As in \cite[Lemma 4.3]{Lombardo2021}, this can all be rephrased in terms of matrices. In the notation of Section \ref{sec:MTnotation}, let $E_i$ denote the CM field of the absolutely simple abelian variety $A_i$ for $1\leq i\leq n$ and let $L$ be a finite Galois extension of $\mathbb Q$ containing the Galois closure of each $E_i$. Furthermore, let $G$ denote the Galois group of $L$ over $\Q$, and $H_i\leq G$ be the subgroup corresponding to $E_i$. 

We let $M_i$ denote the matrix that represents the map  $N_i^* \phi_i^*$. The rows of $M_i$ are indexed by elements of $G$ whereas the columns of $M_i$ are indexed by the set $H_i\setminus G$. We then compute the composition $N_i^*\phi_i^*$ for each $[\beta]\in \widehat{T}_{E_i}$ in order to determine the entries of $M_i$. In the special case where $E_i=L$, the matrix $M_i$ is a square matrix with rows and columns indexed by $G$ (since $H_i=\{1\}$). Its entry in position $(g_j,H_i g_k)$ is given by
\[(g_j,H_i g_k)=\begin{cases}1&\text{if } g_jg_k^{-1}\in \widetilde{\Phi_i^*}\\
0&\text{otherwise.}
\end{cases}
\]

The matrix $M$ representing the map $N_1^* \phi_1^* +\cdots+ N_n^*\phi_n^*$ is obtained by concatenating the matrices $M_1, \ldots, M_n$ horizontally. This construction allows us to easily examine the kernel of the map when working out explicit examples.

We will apply these results to Jacobians of curves of the form $y^2=x^m-1$. In Section \ref{sec:JacobianSplitting} we give the decompositions of the Jacobian varieties into simple factors, and in Section  \ref{sec:DegenerateExamples} we use these decompositions to write the matrix $M$ and study the degeneracy of the Mumford-Tate groups. We see examples where the canonical projection of the Mumford-Tate group is an isomorphism and ones where it is merely an isogeny.

\subsection{Degeneracy via Hodge Rings}\label{sec:DegeneracyHodgeRing}

Recall that the Hodge ring of $A$ is denoted $\mathscr B^*(A)$ and the subring generated by divisor classes is denoted $\mathscr D^*(A)$. As noted in Section \ref{sec:backgroundHodgeRing}, we have the containment $\mathscr D^*(A) \subseteq \mathscr B^*(A)$. The question we wish to answer is: when is the containment strict? More precisely, we wish to know when $\mathscr B^*(A)$ is not generated by the divisor classes $\mathscr B^1(A)$. We can think of this as the Hodge ring being larger than expected. The additional Hodge cycles that are not generated by divisor classes are referred to as {exceptional cycles}, and so we can also ask: in what codimensions do we have exceptional cycles?

In general, these are difficult questions to answer. However, if we specify to Jacobian varieties $J_m$ of curves $C_m\colon y^2=x^m-1$ then we can give satisfactory answers. The following result largely follows from Lemma 5.5 of \cite{Shioda82} and will be used to prove Theorem \ref{thm:degeneracyxm}.

\begin{lemma}\label{lemma:exceptionalHodge}
Let $m=pq$, with $p\leq q$ odd primes. Suppose $q$ is relatively prime to $((p+1)/2)!$. Then  $\mathcal B^d(J_m)$ has exceptional Hodge cycles when $d=(p+1)/2$. 
\end{lemma}

\begin{proof}
Consider the quotient
$$\mathcal B^d(J_m)\bigg/\sum_{r=1}^{d-1} \mathcal B^r(J_m)\cdot \mathcal B^{d-r}(J_m).$$
If the dimension of this space is greater than or equal to 1, then there will be Hodge cycles in $\mathcal B^d(J_m)$ that do not come from lower codimension. In particular, there will be exceptional cycles that are not generated by divisor classes. 

Shioda relates the dimension of this space to the the number $N_m(d)$ of certain indecomposable elements of a semigroup (see Section 1C of \cite{Shioda82}). For certain values of $m$, Shioda gives a lower bound for this quantity (see Lemma 5.5 of  \cite{Shioda82}): letting $m=(2d-1)m'$, we have
\begin{equation}\label{eqn:Shioda_indecomposable}
    N_m(d)\geq m'-1
\end{equation}
whenever $\gcd(m',d!)=1$.

For $m=pq$ and $d=(p+1)/2$, we let $m'=q$ so that $m=(2d -1)m'.$ If  $\gcd(q,d!)=1$ then the inequality in \eqref{eqn:Shioda_indecomposable} becomes
$$N_m(d)\geq q-1$$
which is greater than 1 for any $q>2$. Thus, there are exceptional cycles in $\mathcal B^d(J_m)$, where $d=(p+1)/2$.
\end{proof}
This result will always hold if we choose $q\geq p$ since then $q>(p+1)/2$ and is, therefore, relatively prime to every smaller positive integer. In Section \ref{sec:proofdegeneracyxm} we use this result to prove Theorem \ref{thm:degeneracyxm}.

Lemma \ref{lemma:exceptionalHodge} does not give a value or even a bound for the number of exceptional cycles -- we simply get a lower bound on the number of cycles that do not come from a lower codimension. However, as we will see in Section \ref{sec:DegenerateExamples}, it is possible to compute the generators of the Hodge ring for the Jacobian varieties $J_m$ and so we will be able to see which cycles are exceptional.

\subsection{Degeneracy via Hodge Groups}\label{sec:DegeneracyHodgeGroup}

As noted in Section \ref{sec:backgroundHodgeMT}, the Hodge and Mumford-Tate groups are closely related to each other. The Hodge group is also closely related to the Lefschetz group, denoted $\LL(A)$, which is the connected component of the identity in the centralizer of the endomorphism ring $\End(A_{\overline F})_\Q$ in $\Sp_{V}$. In general, $\Hg(A)\subseteq \LL(A)$ and while the inclusion can be strict, there are cases where we have equality. In fact, when $\dim(A)\leq 3$ it is known that we always have equality \cite{Banaszak2015}. 

The question we wish to answer is: when is the Hodge group smaller than the Lefschetz group? In this paper we study this problem specifically for the Jacobian varieties $\Jac(y^2=x^m-1)$. We could investigate an answer to this through the Hodge group's relationship to the Mumford-Tate group, but we choose a different approach in this paper: we make use of a connection to Hodge rings.

The following result relates the $\Q$-span of Hodge cycles with the Hodge group. Note that, in the notation of Section \ref{sec:DegeneracyHodgeRing}, $\mathscr B^d(A)=H^{2d}_\text{Hodge}(A)\otimes\C$. 

\begin{theorem}\cite[Theorem 17.3.3]{BirkenhakeLange2004}\label{thm:Hodge}
Let $A$ be an abelian variety of dimension $g$. For any $1\leq d \leq g$, denote by 
$$H^{2d}_\text{Hodge}(A):=H^{2d}(A,\Q)\cap H^{d,d}(A)$$
the $\Q$-vector space of Hodge cycles of codimension $d$ on $A$. Then
$$H^{2d}_\text{Hodge}(A)= H^{2d}(A,\Q)^{\Hg(A)}.$$
\end{theorem}

There is a balance between the size of the Hodge group and the dimension of $H^{2d}_\text{Hodge}(A)$. The existence of exceptional Hodge cycles, which we sometimes learn of through Lemma \ref{lemma:exceptionalHodge}, often coincides with a smaller than expected Hodge group. In fact, we see something even more interesting: the exceptional cycles give us information about the embedding of $\Hg(A)$ into $\USp(2g)$. We will see examples of this correspondence between Hodge rings and Hodge groups in Section \ref{sec:DegenerateExamples}.

\subsection{Effects on the Sato-Tate group}\label{sec:DegeneracyST}
As noted in Section \ref{sec:backgroundMTConjecture}, when the Mumford-Tate conjecture holds for an abelian variety $A$ then the Mumford-Tate group and Hodge group can be used to gain information about the Sato-Tate group of $A$. In particular, the Hodge group will equal the identity component of the algebraic Sato-Tate group. The results and exposition in Sections \ref{sec:DegneracyMT}, \ref{sec:DegeneracyHodgeRing}, and \ref{sec:DegeneracyHodgeGroup} help us understand the role degeneracy can play in the Sato-Tate group.

The degeneracy we see in the Mumford-Tate group in Section \ref{sec:DegneracyMT} tells us when the Mumford-Tate group of $A$ coincides with the Mumford-Tate group of certain factors. When the Mumford-Tate conjecture holds for $A$, this will imply that $\ST^0(A)\simeq \ST^0(A_i)$. As with the Mumford-Tate group, in this situation we have a smaller than expected identity component for the Sato-Tate group.

The degeneracy of the Mumford-Tate group gives us some information about the Hodge group and the identity component of the algebraic Sato-Tate group of $A$. However in order to compute moment statistics for the Sato-Tate group, we need to know the explicit embedding of the Hodge group in $\USp(2g)$, where $g$ is the dimension of $A$. The exceptional cycles in the Hodge ring of $A$ give extra relations on the elements of the Hodge ring (see Theorem \ref{thm:Hodge}); these extra relations can be used to determine the embedding of the Hodge group.

Since the Mumford-Tate and Hodge groups only give us information about the identity component of the Sato-Tate group, we will not gain information about the component group. The Sato-Tate group of a degenerate abelian variety  can have not only an identity component that is  smaller than expected, but can also have a larger component group than what would be explained by endomorphisms. In general, for an abelian variety $A/F$ we have a canonical surjection
$$\ST(A)/\ST^0(A) \rightarrow \Gal(K/F),$$
where $K$ is the endomorphism field of $A$, but the surjection is not necessarily an isomorphism if $A$ is degenerate (see, for example, \cite{SutherlandGenus3}).  In general, we have an isomorphism $\ST(A)/\ST^0(A) \simeq \Gal(L/F)$, where $L$ is the minimal Galois extension of $F$ for which $\ST(A_L)$ is connected.\footnote{The field $L$ is exactly the field $K$ when $A$ is nondegenerate. In particular, this is the case for all abelian varieties for dimension $\leq 3$.} The field $L$ is the fixed field of the kernel of a map induced by the $\ell$-adic representation $\rho_{A,\ell}$  in Equation \eqref{eqn:artinrepresntation} (see \cite[Theorem 3.12]{SutherlandNotes}). 
In some cases, we can use results of Zywina \cite{zywina2020determining} to determine this field $L$. In Section \ref{sec:DegenerateExamples} we will see examples where the field $L$ is the endomorphism field of $A$ and ones where an extension is needed.


\section{Decompositions of Jacobians}\label{sec:JacobianSplitting}

Let $C_m/\Q$ denote the smooth projective curve with hyperelliptic equation $y^2=x^m-1$, and let $J_m$ denote its Jacobian. In this section we describe how $J_m$ decomposes into simple factors. We begin with two cases: $m=p^2$ and $m=pq$, where $p$ and $q$ are distinct odd primes.

In the first case, let $\pi_p$ denote the map $\pi_p: (x,y)\mapsto (x^{p},y)$ from $C_m$ to the degree $p$ hyperelliptic curve $C_p\colon y^2=x^p-1$. We let $X$ denote the identity component of the kernel of $(\pi_p)_*\colon J_m \to J_p$. This gives the following isogeny over $\Q$
\begin{align}\label{eqn:Jpp_splitting}
    J_{p^2}\sim X\times J_p.
\end{align}
We can easily compute the dimension of $X$ to be $\dim(J_m)-\dim(J_p)=p(p-1)/2=\phi(m)/2$, where $\phi$ is Euler's totient function.

In the second case, there are two such maps:
 $\pi_p\colon (x,y)\mapsto (x^{q},y)$ from $C_m$ to $C_p:y^2=x^p-1$ and $\pi_q\colon (x,y)\mapsto (x^{p},y)$ from $C_m$ to $C_q:y^2=x^q-1$. We obtain the following isogeny over $\Q$
\begin{align}\label{eqn:Jpq_splitting}
    J_{pq}\sim X\times J_p \times J_q,
\end{align}
where the dimension of $X$ is computed to be $\phi(m)/2$.

\begin{proposition}\label{prop:X_simple_CMtype}
Let $m=p^2$, respectively $pq$. Then the Prym variety $X$ given by the isogeny in Equation \eqref{eqn:Jpp_splitting}, respectively Equation \eqref{eqn:Jpq_splitting}, is absolutely simple. 
\end{proposition}

By work of Shimura-Taniyama \cite{ShimuraTaniyama1961}, we can prove that $X$ is absolutely simple by proving that its CM-type is primitive. Before proving Proposition \ref{prop:X_simple_CMtype}, we provide two lemmas that will be useful.

\begin{lemma}\label{lemma:CMtype}
Let $m=pq$, where $p$ and $q$ are odd primes (not necessarily distinct). Let $X$ be the Prym variety obtained as a quotient of $J_m$ as in \eqref{eqn:Jpp_splitting} or \eqref{eqn:Jpq_splitting}. Then the CM-type of $X$ is $(\Q(\zeta_m); \Phi)$, where
$$\Phi=\{\sigma_j \mid \gcd(j,m)=1\text{ and } j\leq g\},$$ 
$g=(m-1)/2$, and $\sigma_j(\zeta_m)=\zeta_m^j$.
\end{lemma}

\begin{proof}
For ease of notation, let $C\colon=C_m$. Since $C$ is a genus $g=(m-1)/2$ hyperelliptic curve, we may identify the complex uniformization of $J_m$ with $H^0(C_\C, \Omega_C^1)=\C\langle x^idx/y\mid 0\leq i< g\rangle$. We briefly split into two cases:

\textbf{Case 1: $p=q$.} Recall that there is a map 
 $\pi_p: (x,y)\mapsto (x^{p},y)$ from $C$ to $C_p$.  By pulling back regular differentials of $C_p$, we see that $\pi_p^*C_p$ corresponds to the vector subspace $$\C\left\langle{x^{ap-1}}dx/y\mid a=1, 2, \ldots,\frac{p-1}{2}\right\rangle$$ of $H^0(C_\C, \Omega_C^1)$. 
 Thus, the quotient $J_m/J_p$ (which is isogenous to $X$) has natural analytic uniformization given by the quotient of $\C\langle x^idx/y\mid 0\leq i< g\rangle$ by this vector space. The dimension of the new vector space is $\frac{m-1}{2}-\frac{p-1}{2}=\phi(m)/2$.

\textbf{Case 2: $p\not=q$.} In this case, there are two such maps:  $\pi_p: (x,y)\mapsto (x^{q},y)$ from $C$ to $C_p$ and $\pi_q: (x,y)\mapsto (x^{p},y)$ from $C$ to $C_q$. As above, we pull back regular differentials of $C_p$ and $C_q$. The quotient of $J_m$ by the Jacobians of the two curves (which is isogenous to $X$) has  natural analytic uniformization given by the quotient of
$\C\langle x^idx/y\mid 0\leq i< g\rangle$ by the vector subspace generated by
$$\C\left\langle{x^{aq-1}}dx/y\mid a=1, 2, \ldots,({p-1})/{2}\right\rangle \cup \C\left\langle{x^{ap-1}}dx/y\mid a=1, 2, \ldots,({q-1})/{2}\right\rangle.$$
The dimension of the new vector space is $\frac{m-1}{2}-\frac{p-1}{2}-\frac{q-1}{2}=\phi(m)/2$.

In both cases, the automorphism $\alpha\colon(x,y)\mapsto (\zeta_mx,y)$ induces an action of $\Q(\zeta_m)$ on the $(\phi(m)/2)$-dimensional vector space given above. Since $\alpha^*x^jdx/y=\zeta_m^{j+1} x^jdx/y$, we find that characters appearing in this uniformization will be of the form $\zeta_m\mapsto \zeta_m^j$ where $j$ is coprime to $m$. 

Thus, we can identify the CM type $\Phi$ of $\Gal(\Q(\zeta_m)/\Q)$ from the statement of the lemma as the CM-type of  $X$.
\end{proof}

\begin{lemma}\label{lemma:primitive}
Let $m$ be an odd composite number. Define the set 
$$S_m=\left\{1,\ldots, \frac{m-1}{2}\right\} \cap \Z/m\Z^\times.$$
Then $aS_m=S_m$ for some $a\in\Z/m\Z^\times$ if and only if $a=1$.
\end{lemma}
\begin{proof}
First note that $aS_m=S_m$ is clearly true if $a=1$. Also, $aS_m$ cannot equal $S_m$ if $a>\frac{m-1}{2}$ since then $a\cdot 1\not\in S_m$.

Now suppose $a\in S_m$ satisfies $1<a\leq \frac{m-1}{2}$. We will show that there is a value $b\in S_m$ such that $ab\not\in S_m$. 

Note that there is an integer in the interval $\left(\frac{m-1}{2a}, \frac{m-1}{a}\right)$. Indeed, the only way for this to not be true is if $(m-1)/a < (m-1)/2a+1$, which would imply $m-1 < (m-1)/2+a$. However this is not possible since $a\in S_m$ and, hence, $a<(m-1)/2$.

Now let $b$ be an integer in this interval: $\frac{m-1}{2a}<b< \frac{m-1}{a}$. This implies that  $(m-1)/2 <ab< m-1,$
which proves that $ab$ is not in $S_m$. Thus, we have proved the desired result.

\end{proof}

\begin{proof}[Proof of Proposition \ref{prop:X_simple_CMtype}]
As noted below the statement of  Proposition \ref{prop:X_simple_CMtype}, it is sufficient to prove that the CM-type of $X$ is primitive. Proposition 26 of \cite{ShimuraTaniyama1961} claims the following: A CM-type $S$ is primitive if and only if $\gamma S=S$ for $\gamma$ in the Galois group only when $\gamma$ is the identity.

We gave the CM-type $\Phi$ of $X$ in Lemma \ref{lemma:CMtype}. Identifying the automorphisms in $\Phi$ with the set $S_m=\{1,\ldots, (m-1)/2\}\cap \Z/m\Z^\times$ allows us to apply the results of Lemma \ref{lemma:primitive} and conclude that $\gamma \Phi=\Phi$ only when $\gamma$ is the identity element. Hence, $\Phi$ is a primitive CM type and $X$ is absolutely simple. 
\end{proof}

We can see the decompositions in Equations \eqref{eqn:Jpp_splitting} and \eqref{eqn:Jpq_splitting} at the level of Frobenius polynomials. Let $P_{A,\mathfrak p}(x)$ denote the Frobenius polynomial of the abelian variety $A$ at $\mathfrak p$. Then for $A=J_{p^2}$ we have $P_{A,\mathfrak p}(x)=P_{X,\mathfrak p}(x)\cdot P_{J_p,\mathfrak p}(x)$ and for $A=J_{pq}$ we have $P_{A,\mathfrak p}(x)=P_{X,\mathfrak p}(x)\cdot P_{J_p,\mathfrak p}(x)\cdot P_{J_q,\mathfrak p}(x)$ for all good primes $\mathfrak p$.

We can compute Frobenius polynomials for $J_m$ for more general odd, composite $m$ in Sage over good primes $\mathfrak p$. Upon factoring, we are able to see how the Jacobian variety $J_m$ decomposes into simple factors. This leads us to the following conjecture.
\begin{conjecture}
Let $m$ be a positive odd integer. Then the decomposition of $J_m$ into simple factors contains exactly one simple factor of dimension $\phi(d)/2$ for each divisor $d>1$ of $m$. Furthermore, if $p$ is a prime divisor of $m$ then the factor of dimension $\phi(p)/2$ is the Jacobian of the curve $C_p\colon y^2=x^p-1$.
\end{conjecture}

Thus far we have only considered odd values of $m$. In the case where $m$ is even, a result of \cite{EmoryGoodsonPeyrot} can be applied to determine that splitting of the Jacobian. 

\begin{proposition}\label{prop:jacobianpowers}\cite[Lemma 4.1]{EmoryGoodsonPeyrot}
Let $g=2k$ be an even integer, and $C_{2g+2}:y^2= x^{2g+2}+c$, where  $c\in\mathbb Q^{\times}$. Then we have the following isogeny over $\overline{\mathbb Q}$
$$\Jac(C_{2g+2}) \sim \Jac({C_{g+1}})^2,$$
where $C_{g+1}: y^2= x^{g+1}+c$.
\end{proposition}

This allows us to prove the following statement.
\begin{corollary}
Let $m=p^2$ or $pq$. Then the Jacobian of the hyperelliptic curve $C_{2m}$ satisfies the following isogeny over $\overline\Q$
\begin{align*}
J_{2m} \sim \begin{cases}
(X\times J_p)^2 & \text{if } m=p^2,\\
(X\times J_p \times J_q )^2 & \text{if } m=pq.
\end{cases}\end{align*}
\end{corollary}

\begin{proof}
In both cases, $J_{2m}\sim (J_m)^2$ by Lemma 4.1 of \cite{EmoryGoodsonPeyrot} since the dimension of $J_{2m}$ is even: we set $2m=2g+2$ and find that $g=m-1$ is even since $m$ is odd. The result then follows from Equations \eqref{eqn:Jpp_splitting} and \eqref{eqn:Jpq_splitting}.
\end{proof}

\subsection{Proof of  Theorem \ref{thm:degeneracyxm}}\label{sec:proofdegeneracyxm}
We are now in a position to prove that for any odd, composite $m$, the Jacobian $J_m$ of $y^2=x^m-1$ is degenerate.
\begin{proof}
Let $m$ be an odd, composite integer divisible by $pq$, where $p$ and $q$ are primes (not necessarily distinct). As in the the beginning of Section \ref{sec:JacobianSplitting}, we can view the Jacobian of the curve $y^2=x^{pq}-1$ as a factor of $J_m$ via a map on the curve: $\pi_{pq}\colon C_m \to C_{pq}$ defined by $(x,y)\mapsto (x^{m/pq},y)$. We proved in Lemma \ref{lemma:exceptionalHodge} that the Jacobian variety $J_{pq}$ has exceptional Hodge cycles. Thus, $J_m$ has degenerate factors and is, therefore, also degenerate.
\end{proof}


\section{Examples}\label{sec:DegenerateExamples}
In this section we explore the various forms of degeneracy for several Jacobian varieties $J_m$, each illustrating different phenomena that can occur. The examples are presented in a non-hodgepodge manner: we carefully examine degeneracy in the Mumford-Tate groups, the Hodge rings and  Hodge groups, and the Sato-Tate groups of the  varieties. Most of the computations were done in Sage (code available at \cite{GoodsonGithub}).

\subsection{$y^2=x^9-1$}\label{sec:Example9}

We first consider the genus 4 curve $C_9\colon y^2=x^9-1$. The isogeny in Equation \eqref{eqn:Jpp_splitting} tells us that its Jacobian factors as $J_9\sim X \times J_3$. Here $J_3$ denotes the CM elliptic curve $y^2=x^3-1$ and, by Proposition \ref{prop:X_simple_CMtype}, $X$ is an absolutely simple, 3-dimensional abelian variety. Lemma \ref{lemma:CMtype} tells us that the CM field of  $X$ is $\Q(\zeta_9)$. We note that $X$ is not the Jacobian of a curve. Indeed, by results of \cite{Kilicer2018}, (up to isomorphism) there is only one genus 3 curve with CM field $\Q(\zeta_9)$ and primitive CM type: the Picard curve $y^3=x^4-x$. However, by computing and comparing Frobenius polynomials of $y^2=x^9-1$ and $y^3=x^4-x$, we can see that the latter does not appear as a factor of the former.

There are some results in the literature that give us information about the degeneracy of $J_9$. Remark 5.6 of \cite{Lombardo2021} shows that the projection $\MT(J_9)\rightarrow \MT(X)$ is an isomorphism; this is proved using results similar to those in Section \ref{sec:DegneracyMT}. In \cite[Example 6.1]{Shioda82}, Shioda works out the Hodge ring for $J_9$. In particular, Shioda uses a result that is a special case of  Lemma \ref{lemma:exceptionalHodge} to determine that there are exceptional Hodge cycles in codimension 2. The exceptional cycles in codimension 2 can be described by
$$\mathcal B^2(J_9)/\mathcal D^2(J_9)=\{\omega_1\wedge\omega_4\wedge\overline\omega_3\wedge\overline\omega_2, \;\; \omega_2\wedge\omega_3\wedge\overline\omega_4\wedge\overline\omega_1 \},$$
where $\omega_j=x^{j-1}dx/y$ is a standard basis element of $H^1(C_m)$. There are no exceptional cycles in other codimensions.

We now show how we can determine the Hodge group of $J_9$ from this. Recall that the Hodge group $\Hg(J_9)$ is an algebraic subgroup of $\SL_V$ and that $\Hg(J_9)_\C$ contains $h(\mathbb S^1)$. Furthermore, Theorem \ref{thm:Hodge} tells us $H^{2d}_\text{Hodge}(J_9)= H^{2d}(J_9,\Q)^{\Hg(J_9)}.$

We can identify $\SL_V$ with the group of $2g\times2g$ matrices with determinant 1. Since  $J_9$ is of CM-type, Theorem 17.3.5 of \cite{BirkenhakeLange2004} tells us that its Hodge group is commutative. Thus, we can say that elements of $\Hg(J_9)$ are of the form
$$U=\diag(U_1, U_2, U_3, U_4),$$
where $U_i=\diag(u_i,\overline u_i)$ and $u_i\in \mathbb S^1$, and where there may be some relations among the $U_i$.  We compute $U\cdot v$ for each element $v$ of the space of Hodge cycles in order to determine what relations are required in order to obtain $U\cdot v=v$. For example,
\begin{align*}
    U\cdot(\omega_1\wedge\omega_4\wedge\overline\omega_3\wedge\overline\omega_2)&=u_1u_4\overline u_2\overline u_3(\omega_1\wedge\omega_4\wedge\overline\omega_3\wedge\overline\omega_2).
\end{align*}
This yields the relation $u_4=\overline u_1u_2u_3$. The other exceptional cycle yields the same relation. Thus, we may conclude that the elements of $\Hg(J_9)$ are of the form
$$U=\diag(U_1,U_2,U_3,\overline U_1U_2U_3).$$

\subsubsection{Sato-Tate group}

Since the Mumford-Tate conjecture holds for $J_9$, the Hodge group equals the identity component of the algebraic Sato-Tate group. The identity component of the Sato-Tate group is a maximal compact subgroup of this, and so we have the following result. 

\begin{proposition}\label{prop:STidentity9}
Up to conjugation in $\USp(8)$, the identity component of the Sato-Tate group of $J_9$  is isomorphic to $\U(1)^3$. More specifically, 
$$\ST^0(J_9)\simeq\langle\diag(U_1,U_2,U_3,\overline U_1U_2U_3)\rangle.$$
\end{proposition}
This matches the result of Theorem 6.9 of \cite{EmoryGoodsonPeyrot}, though here we used a different method to obtain the result and we give the additional relations needed for the embedding in $\USp(8)$.

In order to determine the full Sato-Tate group of $J_9$, we first need to know the minimal field $L$ for which $\ST((J_9)_L)$ is connected. As noted in Section \ref{sec:DegeneracyST}, this field can be computed using results of Zywina. In fact, Zywina determines that this field is equal to the CM field of $C_9$ (see Section 1.7 of \cite{zywina2020determining}). Hence, $\ST(J_9)/\ST^0(J_9) \simeq \Gal(\Q(\zeta_9)/\Q).$

This is enough information to compare moment statistics for the identity component of the Sato-Tate group to  numerical moments since $\ST^0(J_9)=\ST((J_9)_{\Q(\zeta_9)})$. The numerical $a_1$-moments in Table \ref{table:momemntsJ9zeta9} were computed over the  field $\Q(\zeta_9)$ and for all primes up to $p<2^{23}$ using an algorithm described in \cite{HarveySuth2014} and \cite{HarveySuth2016}. Note that the values of the $a_1$-moments were rounded in order to fit the values in the table.

\begin{table}[h]
\begin{tabular}{|l|l|l|l|l|l|}
\hline
&$M_2$ & $M_4$ & $M_6$ & $M_8$& $M_{10}$\\ 
\hline
$\mu_1$ & 8&216&8000 & 343000& 16003008 \\
$a_1$  &  8.01253&216.204&7997.25&342072& 15901600\\
\hline
\end{tabular}
\caption{Table of $\mu_1$- and $a_1$-moments for $\ST(J_9)$ over $\mathbb Q(\zeta_9)$.}\label{table:momemntsJ9zeta9}\vspace{-.2in}
\end{table}
The errors for the even moments are within a margin of $1\%$ and are consistent with the errors for the odd moments (which should all be zero). Choosing a larger bound for $p$ will lead to better estimates.

For nondegenerate abelian varieties, we can determine generators of the component group of the Sato-Tate group through the twisted Lefschetz group (see \cite{Banaszak2015}). The twisted Lefschetz  group, denoted $\TL(A)$, is a closed algebraic subgroup of  $\Sp_{2g}$ defined by
\begin{align*}
\TL(A):=\bigcup_{\tau \in \Gal(\overline{F}/F)} \LL(A)(\tau),
\end{align*}
where $\LL(A)(\tau):=\{\gamma \in \Sp_{2g}\mid \gamma \alpha \gamma^{-1}=\tau(\alpha) \text{ for all }\alpha \in \End(A_{\overline{F}})_\mathbb{Q}\}$.

We can compute the twisted Lefschetz group for the Jacobian variety $J_9$ even though this is a degenerate abelian variety. The nontrivial Lefschetz sets $\LL(J_9)(\tau)$ are those coming from elements $\tau$ in $G=\Gal(\Q(\zeta_9)/\Q)$ since $\Q(\zeta_9)$ is the endomorphism field of $J_9$. The Galois group $G$ is generated by the order-6 automorphism $\tau_2\colon \zeta_9\to\zeta_9^2$ and we have that the component group of $\bigcup_{\tau \in G}\LL(A)(\tau)$ is generated by 
\begin{align}\label{eqn:J9gamma} \gamma= \begin{pmatrix}0&I&0&0\\0&0&0&I\\0&0&J&0\\J&0&0&0 \end{pmatrix}.
\end{align}
The matrix $\gamma$ was computed using the techniques of \cite{EmoryGoodson2020} and \cite{GoodsonCatalan}.

We now compute moment statistics of the group $\langle\diag(U_1,U_2,U_3,\overline U_1U_2U_3),\; \gamma\rangle$ and compare them to the numerical moments over $\Q$. The numerical moments in Table \ref{table:momemntsJ9Q} were computed for all primes up to $p<2^{23}$ using an algorithm described in \cite{HarveySuth2014} and \cite{HarveySuth2016}. Note that the values of the $a_1$-moments were rounded in order to fit the values in the table.

\begin{table}[h]
\begin{tabular}{|l|l|l|l|l|l|l|}
\hline
&$M_2$ & $M_4$ & $M_6$ & $M_8$& $M_{10}$& $M_{12}$\\ 
\hline
$\mu_1$ & 2& 38&1340 &57190& 2667252& 131481812\\
$a_1$  & 2  &38 &1338 &57010& 2649180& 129958000 \\
\hline
\end{tabular}
\caption{Table of $\mu_1$- and  $a_1$-moments for $J_9$ over $\mathbb Q$.}\label{table:momemntsJ9Q}\vspace{-.2in}
\end{table}
As in Table \ref{table:momemntsJ9zeta9}, the errors for the even moments shown in Table \ref{table:momemntsJ9Q} are within a margin of $1\%$ and are consistent with the errors for the odd moments. This leads us to make the following conjecture.

\conjectureST

As further evidence supporting this claim, we recall that the Sato-Tate group of $J_9$ is connected over the field $\Q(\zeta_9)$ and  note that the component group of the above group is isomorphic to the Galois group $G=\Gal(\Q(\zeta_9)/\Q)$.

\subsubsection{$y^2=x^{18}-1$}\label{sec:Example18}

We first apply Proposition \ref{prop:jacobianpowers} to the Jacobian of the curve $C_{18}\colon y^2=x^{18}-1$ and find that
\begin{align*}
    J_{18}\sim (J_9)^2.
\end{align*}
Hence, by the discussion in Section \ref{sec:backgroundHodgeMT}, the Hodge group $\Hg(J_{18})$ is isomorphic to $\Hg(J_9)$, which acts on $V(J_{18})$ diagonally. Since the Mumford-Tate conjecture holds for $J_9$, it also holds for $J_{18}\sim (J_9)^2$ (see Section \ref{sec:backgroundMTConjecture}). Thus, the Hodge group equals the identity component of the algebraic Sato-Tate group. The identity component of the Sato-Tate group is a maximal compact subgroup of this, and so we have the following result.

\begin{proposition}
Up to conjugation in $\USp(16)$, the identity component of the Sato-Tate group of $J_{18}$ is isomorphic to $(\U(1)^3)_2$. More specifically, 
$$\ST^0(J_{18})=\langle\diag(U_1,U_2,U_3,\overline U_1U_2U_3,U_1,U_2,U_3,\overline U_1U_2U_3)\rangle.$$
\end{proposition}
This matches the result in Section 6.6 of \cite{EmoryGoodsonPeyrot}, though here we give the additional relations needed for the embedding in $\USp(16)$. These additional relations enable us to compute moment statistics for the identity component (see Table \ref{table:momemntsJ18zeta}).
\begin{table}[h]
\begin{tabular}{|l|l|l|l|l|l|l|}
\hline
&$M_2$ & $M_4$ & $M_6$ & $M_8$& $M_{10}$& $M_{12}$\\ 
\hline
$\mu_1$ & 32& 3456 & 512000 & 87808000 & 16387080192 & 3231289442304\\
\hline
\end{tabular}
\caption{Table of $\mu_1$-moments for $J_{18}$ over $\mathbb Q(\zeta_{18})$.}\label{table:momemntsJ18zeta}\vspace{-.2in}
\end{table}

\subsection{$y^2=x^{15}-1$}\label{sec:Example15}

We now consider the curve $C_{15}\colon y^2=x^{15}-1$. The isogeny in Equation \eqref{eqn:Jpq_splitting} tells us that its Jacobian factors as $J_{15}\sim X\times J_5 \times J_3 $, where $X$ is an absolutely simple  4-dimensional abelian variety. By  Lemma \ref{lemma:CMtype}, $X$ has CM by $\Q(\zeta_{15})$. 

While the Hodge ring and Hodge group degeneracies are similar to those we saw in Section \ref{sec:Example9}, we see interesting new phenomena for the Mumford-Tate group and Sato-Tate group. 

\subsubsection{Mumford-Tate group}\label{sec:Example15MT}
In the notation of Section \ref{sec:MTnotation}, let $L=E_1=\Q(\zeta_{15}), E_2=\Q(\zeta_5),$ and $ E_3=\Q(\zeta_3)$. Let $G$ denote the Galois group of $L$ over $\Q$, and $H_i\leq G$ be the subgroup corresponding to $E_i$:  $H_i=\Gal(L/E_i)$. 

The exposition after the proof of Lemma \ref{lemma:MTmaps2} gives a method for representing the map $N_1^* \phi_1^*+N_2^* \phi_2^*+N_3^* \phi_3^*$ as a matrix $M$, which is the concatenation of matrices $M_1,M_2$ and $M_3$ corresponding to the factors of $J_{15}$. The rows of each matrix will be indexed by \begin{align}\label{eqn:Galois15}
G=\{\tau_1, \tau_2, \tau_4, \tau_7,\tau_8,\tau_{11},\tau_{13},\tau_{14}\},    
\end{align}
where $\tau_j(\zeta_{15})=\zeta_{15}^j$. 

For an explicit example, we show how to compute a column for the matrix $M_2$. The CM-type of $E_2=\Q(\zeta_5)$ is $\Phi_2=\{\sigma_1, \sigma_2\}$, where $\sigma_j(\zeta_5)=\zeta_5^j$, and  the reflex type is $\Phi_2^*=\{\sigma_1, \sigma_3\}$. The elements of $H_2=\Gal(L/E_2)$ are $\tau_1$ and $\tau_{11}$. Consider the character $\sigma_2\in\widehat{T}_{E_2}$. We see that
\begin{align*}
    N_2^*\phi_2^*([\sigma_2])&=N_2^*([\sigma_1\sigma_2]+[\sigma_3\sigma_2])\\
                            &=N_2^*([\sigma_2]+[\sigma_1])\\
                            &=([\tau_2]+[\tau_7])+([\tau_1]+[\tau_{11}]).
\end{align*}
Taking into account the labeling of the rows of $M$ coming from Equation \eqref{eqn:Galois15}, we can conclude that the second column of $M_2$ is $(1,1,0,1,0,1,0,0)$.

We repeat this process for  the matrices $M_1, M_2,M_3$ and obtain
\begin{align}\label{eqn:MTmatrix15}
M&=\left(\begin{tabular}{cccccccc|cccc|cc}
1 &1 &1 &1 &0 &0 &0 &0  &1 &1 &0 &0     &1 &0\\
0 &1 &1 &0 &1 &0 &0 &1  &0 &1 &0 &1     &0 &1\\
1& 0& 1& 0& 1& 0& 1& 0  &0& 0& 1& 1     & 1& 0\\
0& 0& 1& 1& 0& 0& 1& 1  & 0& 1& 0& 1    & 1& 0\\
1& 1& 0& 0& 1& 1& 0& 0  & 1& 0& 1& 0    & 0& 1\\
0& 1& 0& 1& 0& 1& 0& 1  & 1& 1& 0& 0    & 0& 1\\
1& 0& 0& 1& 0& 1& 1& 0  & 1& 0& 1& 0    & 1& 0\\
0& 0& 0& 0& 1& 1& 1& 1  & 0& 0& 1& 1    & 0& 1\\
\end{tabular}\right),
\end{align}
where the vertical lines are included to visually separate the $M_i$ matrices. 

\begin{proposition}\label{prop:MT15isogeny}
The canonical projection $\MT(J_{15})\rightarrow X$ is a degree 2 isogeny.
\end{proposition}
\begin{proof}
We use Sage to  compute the (right) kernel of the matrix $M$ in Equation \eqref{eqn:MTmatrix15}. We then determine that the space spanned by kernel of $M$ and $\Z^8\times \{0\}\times \{0\}$ is an index-2 submodule of $\Z^{8}\times\Z^4\times \Z^2$. This corresponds to $\widehat{T}_1\times \{0\}\times\{0\} + \ker (N_1^* \phi_1^* + N_2^*\phi_2^*+ N_3^*\phi_3^*)$ having index 2 in the product $ \widehat{T}_{E_1} \times \widehat{T}_{E_2} \times \widehat{T}_{E_3}$. Thus, by Lemma \ref{lemma:MTmaps2}, the projection is a degree 2 isogeny. 
\end{proof}

\subsubsection{Hodge ring and Hodge group}
Lemma \ref{lemma:exceptionalHodge} tells us that there will be exceptional Hodge cycles in codimension 2. The 12 exceptional cycles in codmension 2 are given by Shioda in \cite[Section 6.2]{Shioda82}. Shioda also identified exceptional Hodge cycles in codimension  3, though in this case we have $\mathscr B^3=\mathscr B^1\cdot\mathscr B^2$.

The Hodge group of $J_{15}$ is commutative, and so we can identify elements $\Hg(J_{15})$ with matrices in $\U(1)^7$. We use the method of Section \ref{sec:Example9} to identify the additional relations on the elements of the Hodge group coming from the  exceptional cycles in the Hodge ring. We find that elements of the Hodge group are of the form
\begin{align*}
    U=\diag(U_1,U_2,U_3,U_4,\overline U_2U_3U_4, \overline U_1U_3U_4, \overline U_1\overline U_2 U_3^2U_4).
\end{align*}
Note that, when identified with a matrix group, the Hodge group is isomorphic to $\U(1)^4$. This is consistent with what we would expect from our work with the Mumford-Tate group in Section \ref{sec:Example15MT} where we determined that $\MT(J_{15})$ is isogenous to the Mumford-Tate  group of a 4-dimensional simple abelian variety.

\subsubsection{Sato-Tate group}
Based on the above work and the relationship between the Sato-Tate group and the Hodge group of an abelian variety, we have the following result.

\begin{proposition}\label{prop:STidentity15}
Up to conjugation in $\USp(14)$, the identity component of the Sato-Tate group of $J_{15}$ is
$$\ST^0(J_{15})=\langle\diag(U_1,U_2,U_3,U_4,\overline U_2U_3U_4, \overline U_1U_3U_4, \overline U_1\overline U_2 U_3^2U_4)\rangle.$$
\end{proposition}
This is consistent with the result in Table 2 of \cite{EmoryGoodsonPeyrot} where it is given that the identity component is isomorphic to $\U(1)^4$. In Proposition \ref{prop:STidentity15} we improve upon the earlier result by giving the additional relations needed for the embedding in $\USp(14)$.

We can easily compute the moment statistics of the Sato-Tate group, but we see an interesting phenomenon when attempting to compare these to the numerical moments coming from the curve $C_{15}$. Recall from Section \ref{sec:DegeneracyST} that, in general, $\ST(A)/\ST^0(A) \simeq \Gal(L/F)$, where $L$ is the minimal Galois extension of of $F$ for which $\ST(A_L)$ is connected. We saw in Section \ref{sec:Example9} that for $J_9$ the field $L$ is exactly the endomorphism field of the variety. However that is not the case for $J_{15}$ -- we need a degree 2 extension of $\Q(\zeta_{15})$.

We get our first hint that we may need a larger field by examining the moment statistics in Table \ref{table:momemntsJ15zeta15}. The numerical moments were computed over the field $\Q(\zeta_{15})$ for primes $p<2^{32}$ by Sutherland \cite{SutherlandEmail} using an algorithm described in \cite{HarveySuth2014} and \cite{HarveySuth2016}. Note that the values of the $a_1$-moments were rounded in order to fit the values in the table.

\begin{table}[h]
\begin{tabular}{|l|l|l|l|l|l|}
\hline
& $M_2$ & $M_3$&$M_4$ & $M_5$ & $M_6$\\ 
\hline
$\mu_1$& 14 &  0 & 834 &  0 & 78260\\
$a_1$ & 13.992&  0.037  &641.326 &   5.495 & 49354.840\\
\hline
\end{tabular}
\caption{Table of $\mu_1$- and $a_1$-moments for $\ST(J_{15})$ over $\mathbb Q(\zeta_{15})$.}\label{table:momemntsJ15zeta15}\vspace{-.2in}
\end{table}
We quickly see that these moment statistics do not match -- the errors for $M_4$ and $M_6$ are much larger than we would expect for the Sato-Tate group of the abelian variety.  

We can use a technique developed in \cite{zywina2020determining} to gain information about the field $L$ -- we will describe the strategy as it applies to the abelien variety $J_{15}$. For a prime $p\equiv 1 \pmod{15}$, we form the group $\Phi_{J_{15},p}$ generated by the set of roots in $\overline \Q$ of the Frobenious polynomial of $J_{15}$. If the group $\Phi_{J_{15},p}$ is torsion-free then $p$ splits completely in $L$ (see Section 1.5 of \cite{zywina2020determining}). By computing this group for many primes, it appears that roughly half of the primes congruent to 1 modulo 15 split completely in $L$  (there is Magma code to assist with these computations available at \cite{ZywinaGithub}). This indicates that $L$ is a degree 2 extension of the CM field $\Q(\zeta_{15})$. 

We use the L-functions and Modular Forms Database \cite{lmfdb} to gain further information about the field $L$. There is only one field listed that is a degree 2 extension of $\Q(\zeta_{15})$ with the same ramified primes (3 and 5):
\begin{center}
    \url{http://www.lmfdb.org/NumberField/16.0.3243658447265625.1}
\end{center}
In order to further confirm that this is the field $L$ for which $\ST(A_L)$ is connected, we compute numerical moment statistics. The numerical moments in Table \ref{table:momemntsJ15L} were computed over the above field for primes $p<2^{32}$ by Sutherland \cite{SutherlandEmail} using an algorithm described in \cite{HarveySuth2014} and \cite{HarveySuth2016}. Note that the values of the $a_1$-moments were rounded in order to fit the values in the table.

\begin{table}[h]
\begin{tabular}{|l|l|l|l|l|l|}
\hline
& $M_2$ & $M_3$&$M_4$ & $M_5$ & $M_6$\\ 
\hline
$\mu_1$& 14 &  0 & 834 &  0 & 78260\\
$a_1$ & 13.993  &  0.093  &  833.023   &  11.991  &  78067.503\\
\hline
\end{tabular}
\caption{Table of $\mu_1$- and $a_1$-moments for $\ST(J_{15})$ over $L$.}\label{table:momemntsJ15L}\vspace{-.2in}
\end{table}
We see that the errors for the even moments are within a margin of $1\%$ and are consistent with the errors for the odd moments.

Further work is needed in order to determine the full Sato-Tate group of  $J_{15}$ (over $\Q$). We know that the component group of the Sato-Tate group will satisfy $\ST(J_{15})/\ST^0(J_{15}) \simeq \Gal(L/\Q)$, however it is not clear how one could find explicit generators for the component group.

\subsection{$y^2=x^{21}-1$}\label{sec:Example21}

For our third example we consider the curve $C_{21}\colon y^2=x^{21}-1.$ The isogeny in Equation \eqref{eqn:Jpq_splitting} tells us that its Jacobian factors as $J_{21}\sim X\times J_7 \times J_3 $, where $X$ is an absolutely simple abelian variety of dimension 6. By  Lemma \ref{lemma:CMtype},  $X$ has CM by $\Q(\zeta_{21})$. In this case we see new phenomena in both the Hodge ring setting, as well as for the Mumford-Tate group.

\subsubsection{Mumford-Tate Group}

In order to study canonical projections of the Mumford-Tate group, we build the matrix $M$ representing the map $N_1^* \phi_1^*+N_2^* \phi_2^*+N_3^* \phi_3^*$ described in Section \ref{sec:canonicalprojmap} (see Section \ref{sec:Example15MT} for more details on how to build this matrix). The matrix $M$ is obtained by concatenating matrices $M_1, M_2,M_3$ corresponding to the factors $X, J_7, J_3$, respectively:

\begin{align}\label{eqn:MTmatrix21}
M=\left(\begin{array}{rrrrrrrrrrrr|rrrrrr|rr}
1 & 1 & 1 & 1 & 1 & 1 & 0 & 0 & 0 & 0 & 0 & 0 & 1 & 1 & 1 & 0 & 0 & 0 & 1 & 0 \\
0 & 1 & 1 & 0 & 1 & 1 & 0 & 0 & 1 & 0 & 0 & 1 & 0 & 1 & 0 & 1 & 0 & 1 & 0 & 1 \\
0 & 0 & 1 & 0 & 1 & 0 & 1 & 0 & 1 & 0 & 1 & 1 & 1 & 0 & 0 & 1 & 1 & 0 & 1 & 0 \\
0 & 0 & 1 & 1 & 1 & 1 & 0 & 0 & 0 & 0 & 1 & 1 & 1 & 0 & 1 & 0 & 1 & 0 & 0 & 1 \\
1 & 0 & 0 & 0 & 1 & 0 & 1 & 0 & 1 & 1 & 1 & 0 & 1 & 1 & 1 & 0 & 0 & 0 & 0 & 1 \\
0 & 0 & 0 & 0 & 1 & 1 & 0 & 0 & 1 & 1 & 1 & 1 & 0 & 1 & 1 & 0 & 0 & 1 & 1 & 0 \\
1 & 1 & 1 & 1 & 0 & 0 & 1 & 1 & 0 & 0 & 0 & 0 & 1 & 0 & 0 & 1 & 1 & 0 & 0 & 1 \\
0 & 1 & 1 & 1 & 0 & 1 & 0 & 1 & 0 & 0 & 0 & 1 & 0 & 0 & 0 & 1 & 1 & 1 & 1 & 0 \\
1 & 1 & 0 & 0 & 0 & 0 & 1 & 1 & 1 & 1 & 0 & 0 & 0 & 1 & 0 & 1 & 0 & 1 & 1 & 0 \\
1 & 1 & 0 & 1 & 0 & 1 & 0 & 1 & 0 & 1 & 0 & 0 & 0 & 1 & 1 & 0 & 0 & 1 & 0 & 1 \\
1 & 0 & 0 & 1 & 0 & 0 & 1 & 1 & 0 & 1 & 1 & 0 & 1 & 0 & 1 & 0 & 1 & 0 & 1 & 0 \\
0 & 0 & 0 & 0 & 0 & 0 & 1 & 1 & 1 & 1 & 1 & 1 & 0 & 0 & 0 & 1 & 1 & 1 & 0 & 1
\end{array}\right),
\end{align}
where the vertical lines are included to visually separate the $M_i$ matrices. We use Sage to  compute the (right) kernel of the matrix $M$ in Equation \ref{eqn:MTmatrix21}. The space spanned by the kernel of $M$ and $\Z^{12}\times \{0\}\times \{0\}$ does not have finite index in $\Z^{12}\times\Z^6\times \Z^2$, which tells us that the projection $\MT(J_{21})\to \MT(X)$ is not an isomorphism nor an isogeny. However, we do find that the space spanned by $\ker(M)$ and $\Z^{12}\times \{0\}\times \Z^{2}$  has index 1 in $\Z^{12}\times\Z^6\times \Z^2$. This yields the following result.
\begin{proposition}
The projection $\MT(J_{21})\to \MT(X\times J_3)$ is an isomorphism.
\end{proposition}

\subsubsection{Hodge Ring and Hodge Group}

The exceptional Hodge cycles for $J_{21}$ are given in \cite[Section 6]{Shioda82}. Applying the technique described in Section \ref{sec:Example9}, we find that the elements of  Hodge group are of the form  
\begin{align*}
    \diag(U_1,\,U_2,\,U_3,\,\overline U_1U_2U_3,\,U_5,\, U_6, U_7,\,\overline U_1U_3U_6,\, \overline U_2 U_5U_6, \,\overline U_1U_5U_6).
\end{align*}
This implies that the Hodge group is isomorphic to $\U(1)^6$, which is surprising since the Mumford-Tate group of $J_{21}$ is isomorphic to the Mumford-Tate group of a 7-dimensional subvariety. It is possible that this phenomenon is related to the fact that the abelian variety $X$ is degenerate. This fact is proven by Shioda in \cite[Section 6]{Shioda82} by noting that there are exceptional cycles of $J_{21}$ that come from $X$. This differs from our examples for $m=9, 15$, where in those cases all of the exceptional Hodge cycles came from the Jacobian factors of $J_m$.

\subsection{$y^2=x^{27}-1$}\label{sec:Example27}

An interesting example to consider next is the genus 13 curve $C_{27}\colon y^2=x^{27}-1$, which is not one of the examples worked out by Shioda. Its Jacobian factors as $J_{27}\sim X_2 \times X_1 \times J_3$, where $X_1$ and $J_3$ are the simple abelian varieties appearing in the decomposition of $J_9$ in Section \ref{sec:Example9} and $X_2$ is an absolutely simple abelian variety. To find the dimension of $X_2$ we simply compute $\dim(J_{27})-\dim(J_9)=9$. We can easily extend Lemma \ref{lemma:CMtype} and conclude that the CM field of $X_2$ is $\Q(\zeta_{27})$. 

\subsubsection{Mumford-Tate Group}
Using the same technique described in detail in Section \ref{sec:Example15MT}, we are able to prove the following result.

\begin{proposition}\label{prop:MT27isomorphism}
The canonical project $\MT(J_{27}) \to \MT(X_2)$ is an isomorphism.
\end{proposition}
\begin{proof}
We use Sage to  compute the (right) kernel of the matrix $M$ obtained using the techniques described in Section \ref{sec:canonicalprojmap}. The order of  character group $\widehat{T}_E$ associated to $X_2$ is $\phi(27)=18$. We determine that the space spanned by kernel of $M$ and $\Z^{18}\times \{0\}\times \{0\}$ spans all of  $\Z^{18}\times\Z^6\times \Z^2$. Thus, by Lemma \ref{lemma:MTmaps2}, the projection is an isomorphism. 
\end{proof}

\subsubsection{Hodge Ring and Hodge Group}

The Hodge ring of $J_{27}$ is not worked out in Shioda's paper \cite{Shioda82}, but we can use Theorem 5.2 of the paper to determine the exceptional cycles that appear in the Hodge ring. For example, using code written in Sage, we find that the space of exceptional cycles in codimension 2 is generated by:
\begin{align*}
    \omega_1\wedge\omega_{10}\wedge\overline\omega_3\wedge\overline\omega_8 &&\omega_3\wedge\omega_8\wedge\overline\omega_1\wedge\overline\omega_{10}  \\
    \omega_2\wedge\omega_{11}\wedge\overline\omega_6\wedge\overline\omega_7 &&\omega_6\wedge\omega_7\wedge\overline\omega_2\wedge\overline\omega_{11}\\
    \omega_3\wedge\omega_{12}\wedge\overline\omega_6\wedge\overline\omega_{9}
    &&\omega_6\wedge\omega_{9}\wedge\overline\omega_3\wedge\overline\omega_{12}\\
    \omega_4\wedge\omega_{13}\wedge\overline\omega_5\wedge\overline\omega_{12}  &&\omega_5\wedge\omega_{12}\wedge\overline\omega_4\wedge\overline\omega_{13}
\end{align*}

As in Sections \ref{sec:Example9} and \ref{sec:Example15}, we can use the exceptional Hodge cycles to determine extra relations in the Hodge group of $J_{27}$ from this. We find that  the elements of $\Hg(J_{27})$, identified as $2g\times 2g$ matrices with determinant 1, are of the form
$$\diag(U_1,U_2,U_3,U_4,U_5,U_6,U_7,U_8,U_9,\overline U_1U_3U_8,\overline U_2U_6U_7,\overline U_3U_6U_9,\overline U_3\overline U_4U_5U_6U_9).$$
Note that this implies that $\Hg(J_{27})$ is isomorphic to $\U(1)^9$ with embedding in $\USp(26)$ given by the relations noted above.

\subsection{Further work for more general $y^2=x^m-1$}

In this section we will investigate further examples and make conjectures about the degeneracy of  Jacobian varieties $J_m$. We first consider examples where $m=p^2$, for $p$ prime.

Recall from Section \ref{sec:JacobianSplitting} that the Jacobian of $C_{p^2}$ satisfies the following isogeny
$$J_{p^2} \sim  X \times J_p,$$
where $X$ is an absolutely simple abelian variety of dimension $p(p-1)/2$.

The following result has been verified by computing the kernel of the matrix $M$ described in Section \ref{sec:canonicalprojmap}.
\begin{proposition}\label{prop:MTp2}
Let $J_{p^2}$ and $X$ be defined as above. Then the projection $\MT(J_{p^2})\rightarrow \MT(X)$ is an isomorphism for odd primes $p\leq 29$.
\end{proposition}

It seems likely that this holds more generally, and so we make the following conjecture.
\begin{conjecture}\label{conjec:MTp2}
Let $J_{p^2}$ and $X$ be defined as above. Then the projection $\MT(J_{p^2})\rightarrow \MT(X)$ is an isomorphism.
\end{conjecture}

The next generalization we consider is $m=p^3$. As in Section  \ref{sec:Example27} we have that the Jacobian $J_{p^3}$ satisfies the following isogeny
$$J_{p^3} \sim  X_2 \times X_1 \times J_p,$$
where $X_1$ and $J_p$ are the simple abelian varieties appearing in the decomposition of $J_{p^2}$ and $X_2$ is an absolutely simple abelian variety. To find the dimension of $X_2$ we simply compute $\dim(J_{p^3})-\dim(J_{p^2})$. 

The following result is verified using the same techniques as those in Section \ref{sec:Example27}.
\begin{proposition}\label{prop:MTp3}
Let $J_{p^3}$ and $X_2$ be defined as above. Then the projection $\MT(J_{p^3})\rightarrow \MT(X_2)$ is an isomorphism for odd primes $p\leq 13$.
\end{proposition}

As we suspect for $m=p^2$, it seems likely that this holds more generally, and so we make the following conjecture.
\begin{conjecture}\label{conjec:MTp3}
Let $J_{p^3}$ and $X_2$ be defined as above. Then the projection $\MT(J_{p^3})\rightarrow \MT(X_2)$ is an isomorphism.
\end{conjecture}

We believe that these results could be generalized to $m=p^k$, for $k>1$. For values of $m$ with at least two distinct odd prime factors, it is more difficult to make a general statement regarding the canonical projection of the Mumford-Tate group onto its factors. It remains an open question of when the projection onto one or more factors is an isogeny or isomorphism.

In terms of the Hodge rings and Hodge groups,  Lemma \ref{lemma:exceptionalHodge} tells us that there are exceptional Hodge cycles in codimension $(p+1)/2$ for the Jacobian $J_m$, though there may be exceptional cycles in other codimensions as well. One can compute all of the Hodge cycles using Sage code that implements Theorem 5.2 of \cite{Shioda82} and then determine which of these are exceptional. These exceptional cycles are used to determine extra relations in the Hodge ring, which in turn give us the embedding of the identity component of the Sato-Tate group of $J_m$ in $\USp(2g)$. In the examples worked out in the preparation of this paper, there was not a clear pattern in the number or the format of the exceptional cycles. It would be interesting to work out more examples and determine if there is a way to characterize the exceptional cycles based on the factorization of $m$.

One of the initial goals for this project was to better understand the Sato-Tate groups of degenerate abelian varieties. By determining the additional relations in the Hodge group (coming from exceptional cycles in the Hodge ring), we can describe the identity component of the Sato-Tate group. Beyond this, there is still much work to be done. The first main question to consider further is: over what extension $L/\Q$ is the Sato-Tate group connected? We saw that in the case of $J_9$, $L$ is exactly the endomorphism field, but for $J_{15}$ it is a degree 2 extension of this field. Can we make a general statement about the field $L$?

The second main question is: how do we find generators of the component group of the Sato-Tate group? This is particularly unclear in the situation where the field $L$ strictly contains the endomorphism field, which also means that any results in this direction will be particularly interesting.

\section*{Acknowledgements}
The author would like to thank the anonymous reviewers for their helpful comments on an earlier draft of this article. The author also thanks Drew Sutherland and Davide Lombardo for their helpful conversations while working on Sections \ref{sec:Example9} and \ref{sec:Example15}, and Maria Fox for her feedback on an earlier draft of the Introduction.

The author was partially supported by NSF grant DMS-2201085 and by a PSC-CUNY Award, jointly funded by The Professional Staff Congress and The City University of New York. This project began while the author was a visiting research scientist at the Max-Planck-Institut f\"ur Mathematik. She is grateful for the supportive research environment they provided.

\bibliographystyle{abbrv}
\bibliography{Hodgebib}

\end{document}